\documentclass[a4paper]{article}

\usepackage{amsthm}
\usepackage{hyperref}
\usepackage{amsmath}
\usepackage{amssymb}

\newcommand{\norm}[1]{\left\lVert #1 \right\rVert}

\newcommand{\hati}{\hat{I}}

\newtheorem{theorem}{Theorem}[section]
\newtheorem{lemma}{Lemma}[section]
  
\newtheorem{corollary}{Corollary}[section]  

\begin{document}

\title{An Explicit Upper Bound for Modulus of 
    Divided Difference on A Jordan Arc in the Complex Plane}
\author{Difeng Cai}


\maketitle

\begin{abstract}
    An explicit upper bound is derived for the modulus of divided difference for a smooth(not necessarily analytic) function
    defined on a smooth Jordan arc (or a smooth Jordan curve) in the complex plane.
    As an immediate application,
    an error estimate for complex polynomial interpolation on 
    a Jordan arc (or a Jordan curve) 
    is given, 
    which extends the well-known error estimate for polynomial interpolation on the unit interval.
    Moreover, this upper bound is independent of the parametrization of the curve.
\end{abstract}

\section{Introduction} 
\label{sec:intro}
Suppose $f$ is a smooth function on $[0,1]$.
The problem of interpolating $f$ at $n+1$ distinct nodes $z_1,\dots,z_{n+1}\in [0,1]$ 
using a polynomial of degree $n$ has a satisfying answer,
for which we have the following well-known error estimate (cf. \cite[p.314]{ascher})
\begin{equation}
    \label{eq:error0}
    \left\lvert f(z)-p_n(z)\right\rvert \leq \dfrac{\sup_{0<\xi<1}\left\lvert f^{(n+1)}(\xi)\right\rvert}{(n+1)!} 
    \prod_{k=1}^{n+1} \left\lvert z-z_k\right\rvert.
\end{equation}
However, if we consider polynomial interpolation with a complex variable
for a smooth non-analytic function $f$ defined on a smooth Jordan arc or a smooth Jordan curve in the complex plane,
no result that resembles \eqref{eq:error0} is available.

Even though efforts have been made over the years in complex polynomial interpolation,
for example,
with monographs on this topic by J.L.Walsh (\cite{walsh}), D.Gaier (\cite{gaier}), etc.,
and numerous papers 
such as \cite{hermite1878}, \cite{fejer},
\cite{curtisslagrange}, \cite{curtisscondition}, \cite{curtisscircle}, \cite{curtissjordan},
\cite{reichelconformal}, \cite{reichelnewton}, etc.,
the number of literatures investigating 
a possible extension of \eqref{eq:error0} to the complex plane
is quite limited.
Moreover, most of the results in existing literatures on complex polynomial interpolation
require $f$ be analytic in certain domain of interest
(cf. \cite{hermite1878}, \cite{gaier}, \cite{curtisslagrange}, \cite{reichelconformal}, \cite{reichelnewton}),
or the curve be analytic (cf.\cite{curtisslagrange}),
and all of them focus on interpolation on a boundary curve (instead of a piece of arc)
due to various needs,
for example,
in conformal mapping (cf.\cite{reichelconformal}) and
in solving Dirichlet problems (cf.\cite{curtissdirichlet}), etc.

In terms of extending \eqref{eq:error0} to the complex plane,
we note that since \eqref{eq:error0} can be deduced by estimating divided difference
for $f$,
it then boils down to estimate the divided difference for the general case 
where $f$ is not necessarily analytic and 
the Jordan arc (or Jordan curve) is not analytic, either.
The only paper we can find that deals with this issue is \cite{curtissjordan}, 
in which the author showed the uniform boundedness of 
the divided difference for $f$ on a Jordan curve 
in the complex plane. 
Though no explicit bound was given in \cite{curtissjordan},
by following the proof in \cite{curtissjordan},
we are able to find an upper bound that scales like $C_f a^{n^2}/n!$,
where $C_f$ is a constant related to certain derivative of $f$ 
and $a > 1$ is a constant depending on the parametrization of the curve.
Obviously, this bound is too large to use in practice.
The appearance of $a^{n^2}$ in the estimate is due to the \emph{indirect} approach 
used in \cite{curtissjordan} to bound the divided difference, 
where the problem for a general Jordan curve was transformed into the problem for a unit circle
by a change of variable (cf.\cite[Lemma 3.1, Lemma 3.3]{curtissjordan}), leading to 
an estimate highly sensitive to the parametrization of the curve.
Therefore, it is the aim of this paper to employ a \emph{direct} approach to provide an upper bound 
independent of the parametrization of the curve and hopefully
in a similar form as in the real case. 
As a straightforward application,
an error estimate for polynomial interpolation with a complex variable 
on a Jordan arc or a Jordan curve in the complex plane 
will be obtained, 
which can be viewed as an extension of \eqref{eq:error0}.

Since estimates for divided differences on Jordan curves as in \cite{curtissjordan} 
can all be derived by transforming the problem into 
estimating divided differences on Jordan arcs,
it suffices for us to focus only on divided differences on Jordan arcs,
and the case for Jordan curves can be immediately obtained
as a byproduct.

We define divided difference as follows.
For $n+1$ distinct points 
$z_1, \dots, z_{n+1}$
on the complex plane, and a function $f$ defined on a set containing those points,
the divided difference for $f$ of order $k(k=0,\dots, n)$ with respect to those points
is defined recursively by 
\begin{equation}
\label{eq:dvdef}
\begin{aligned}
    d_0 &= d_0(f|z_1) = f(z_1) \\
    d_k &= d_k(f|z_1,\dots,z_{k+1}) \\
        &= \dfrac{d_{k-1}(f|z_1,z_2,\dots,z_{k})-
d_{k-1}(f|z_{k+1},z_2,\dots,z_{k})}{z_1-z_{k+1}}, \quad 1\leq k \leq n.
\end{aligned}
\end{equation}

It will be shown later how to define divided difference properly at those points when 
$z_i=z_j$ for some $i\neq j$.

There are several definitions or representations for divided difference.

Recall the Newton divided difference interpolation formula (cf.\cite{norlund},
\cite{ascher})
\begin{multline}
    \label{eq:newton}
    p_n(z) = d_0(f|z_1) + d_1(f|z_1,z_2)(z-z_1) + d_2(f|z_1,z_2,z_3)(z-z_1)(z-z_2) \\
+\dots + d_n(f|z_1,\dots,z_{n+1})(z-z_1)(z-z_2)\dots(z-z_n),
\end{multline}
or in the form (cf.\cite{norlund})
\begin{equation}
    \label{eq:newton1}
    f(z) = p_n(z) + d_{n+1}(f|z_1,\dots,z_{n+1},z)(z-z_1)(z-z_2)\dots(z-z_{n+1}),
\end{equation}
with $p_n(z)$ given in \eqref{eq:newton},
and the Lagrange interpolation formula (cf.\cite{norlund},
\cite{ascher})
\begin{equation}
    \label{eq:lagrange}
    p_n(z) = \sum_{k=1}^{n+1} f(z_k)\dfrac{w_k(z)}{w_k(z_k)},
\end{equation}
where 
\[
    w_k(z) := \prod_{i\neq k} (z-z_i).
\]
Since the interpolating polynomial of degree at most $n$
is unique, by comparing the coefficient of $z^n$ in \eqref{eq:newton} and \eqref{eq:lagrange},
we see that 
\begin{equation}
\label{eq:dvlag}
    d_n(f|z_1,\dots,z_{n+1}) = \sum_{k=1}^{n+1} \dfrac{f(z_k)}{w_k(z_k)}.
\end{equation}
In addition to \eqref{eq:dvdef},
Eq.\eqref{eq:dvlag} above provides another definition for divided difference,
from which it can be seen that the divided difference
$d_n(f|z_1,\dots,z_{n+1})$ is invariant under 
any permutation of interpolation nodes $z_1,\dots,z_{n+1}$.

Besides, suppose either $f$ is analytic in a neighborhood of the convex hull of 
$S=\{z_1,\dots,z_{n+1}\} \subset \mathbb{C}$, 
or $f$ is smooth in an open interval containing $S$
if $S\subset \mathbb{R}$,
an integral representation known as Hermite formula 
(or Genocchi-Hermite formula, cf.\cite{deboordvdiff})
can be derived (cf.\cite{hermite1878}, \cite{norlund}):
\begin{multline}
\label{eq:dvintegral}
    d_n(f|z_1,\dots,z_{n+1}) = \int_0^1 dt_1 \int_0^{t_1} dt_2\dots \int_0^{t_{n-1}} 
    f^{(n)} \Big( (1-t_1)z_1+(t_1-t_2)z_2+\dots\\
    +(t_{n-1}-t_n)z_n+t_n z_{n+1}\Big) dt_n  .
\end{multline}
From this representation, it is straightforward to obtain an upper bound of $|d_n|$,
namely,
\begin{equation}
\label{eq:bound0}
    |d_n(f|z_1,\dots,z_{n+1})| \leq \dfrac{\norm{f^{(n)}}_{\text{sup}}}{n!},
\end{equation}
where the \emph{sup} of $|f^{(n)}|$ is taken inside the convex hull of $S$.

However, if the assumptions above on $f$ and $S$ do not hold, 
namely, if $S\not\subset \mathbb{R}$ and $f$ is not analytic, 
no result on the upper bound of $|d_n|$
can be found other than the one in \cite{curtissjordan}.
As mentioned before, 
the result in \cite{curtissjordan} is too pessimistic and is dependent on the parametrization
of the curve. 
We aim to find an upper bound for $|d_n|$ that is independent of the parametrization of the curve,
and is in a similar form as the one in \eqref{eq:bound0}.
Namely,
we are looking for an upper bound 
${C_f \alpha(n)}/{n!}$, where $C_f$ is a positive constant depending on certain derivatives 
(in the sense of \eqref{eq:deriv} described later) of $f$ and 
$\alpha(n)$ is a positive function of $n$ that grows at most exponentially in $n$,
i.e., $\alpha(n) \leq a b^n$ for some constant $a, b > 0$
(hence $\lim_{n\to\infty} {\alpha(n)}/{n!}=0$),
and $\alpha(n)$ is independent of the parametrization of the curve.

The rest of this paper is organized as follows.
In Section \ref{sec:property}, we present some technical tools developed in \cite{curtissjordan} 
that will be used in our proof.
In Section \ref{sec:bound}, we derive an upper bound for modulus of divided difference
of a smooth function defined on a Jordan arc (or a Jordan curve) in the complex plane,
and thus establish an explicit error estimate for complex polynomial interpolation on 
a Jordan arc (or a Jordan curve).

\section{Properties of divided difference as a function of one variable} 
\label{sec:property}
By a \emph{Jordan arc}, we mean the image of the closed unit interval $[0,1]$
under a homeomorphism into the complex plane,
while a \emph{Jordan curve} on the complex plane is the homeomorphic image of the unit circle.
We say a Jordan arc (or a Jordan curve) is \emph{admissible} 
if it is the image of a parametrization $\phi$ where
$\phi\in C^1([0,1])$ and $\phi'(t)\neq 0$ for all $t\in [0,1]$,
where at endpoints $0$ and $1$,
$\phi'(0)$ and $\phi'(1)$ are interpreted as one-sided limits 
(additionally, $\phi(0)=\phi(1), \phi'(0)=\phi'(1)$ for a Jordan curve).

As was developed in \cite{curtissjordan},
first we need to define derivatives of a function on the arc.

For a function $f$ defined on the Jordan arc $\gamma$,
the (first order) derivative of $f$ at $z_1\in \gamma$ is given by 
\begin{equation}
    \label{eq:deriv}
    f'(z_1) = \lim_{z\to z_1} \dfrac{f(z)-f(z_1)}{z-z_1}, \; z\in\gamma,
\end{equation}
as long as the limit exists.
Inductively, higher order derivatives can be defined.
We shall use $f^{(k)}$ to denote the $k^{\text{th}}$ order derivative of $f$.
As an example,
we assume that $\gamma$ is an admissible Jordan arc parametrized by $\phi:[0,1]\to \mathbb{C}$
with $\phi'(t)\neq 0$, $\forall t\in [0,1]$,
and that $f\circ\phi\in C^1([0,1]).$ 
The definition in \eqref{eq:deriv} then leads to 
\begin{equation}
    \label{eq:smoothderiv}
    f'(z_1)  = (f\circ\phi)'(t_1) / \phi'(t_1),
\end{equation}
where $\phi(t_1)=z_1.$
Moreover, 
the integral of $f$ on $\gamma$ from $z_2=\phi(t_2)$ to $z_1=\phi(t_1)$ is given by 
\begin{equation}
\label{eq:inte}    
    \int_{z_2}^{z_1} f(z) dz = \int_{t_2}^{t_1} f\circ\phi(t)\, \phi'(t) dt,
\end{equation}
which is independent of the parametrization of $\gamma$ (cf.\cite[p.21]{steincomplex}).
If $f\circ\phi (t)$ is absolutely continuous on $[0,1]$,
it can be easily seen that the derivatives in Eq.\eqref{eq:smoothderiv} 
exist almost everywhere in $t_1\in[0,1]$, 
and we have below the fundamental theorem of calculus for $f$ defined on $\gamma$
\begin{equation*}
    \int_{z_2}^{z_1} f'(z) dz = \int_{t_2}^{t_1} (f\circ\phi)'(t) dt = 
    f\circ\phi(t_1)-f\circ\phi(t_2) = f(z_1)-f(z_2).
\end{equation*}

Having established the calculus, we then present some properties of divided difference
as a function of one variable, all of which can be immediately obtained from results 
in \cite{curtissjordan}.

We follow the notation used in \cite{curtissjordan}, where 
the $k^{\text{th}}$ order partial derivative of the divided difference
$d_n(f|z_1, \dots, z_{n+1})$
with respect to $z_1$ is denoted by $d_n^k$, namely,
\[
d_n^k = d_n^k(f|z_1,\dots,z_{n+1}) = \dfrac{\partial^k d_n(f|z_1,\dots,z_{n+1})}{\partial z_1^k}.
\]

The following lemma is same to Lemma $3.2$ in \cite{curtissjordan}
by simply replacing "Jordan curve" in \cite{curtissjordan} with "Jordan arc" in our setting,
and the proof is almost the same.
\begin{lemma}
\label{lm:lemma1}
    Let $f$ be a function defined on an admissible Jordan arc $\gamma$ with parametrization $\phi$
    such that $f^{(n-1)}$ exists everywhere on $\gamma$ 
    and $f^{(n-1)}\circ \phi (t)$ is absolutely continuous on $[0,1]$. 
    Then we have 
    \begin{equation}
        \label{eq:lemma1}
        d_1^{k-1}(f|z_1,z_2) = \dfrac{\int_{z_2}^{z_1} (z-z_2)^{k-1} f^{(k)}(z)dz}{(z_1-z_2)^k},
        \; k=1,2,\dots,n,
    \end{equation}
    where $z_1,z_2\in\gamma, \text{ and } z_1\neq z_2.$
\end{lemma}

In order to derive the smoothness of divided difference for a smooth function,
we need the following boundedness result,
whose proof can be found in \cite[Lemma 3.3]{curtissjordan}.
\begin{lemma}
    \label{lm:lemma2}
    Let $f$ be a function defined on an admissible Jordan arc $\gamma$ 
    such that $|f\circ\phi|$ is uniformly bounded on $[0,1]$ except on a set of measure zero.
    We define 
    \[
        J_k(z_1,z_2) = \int_{z_2}^{z_1} (z-z_2)^k f(z)dz,\quad z_1,z_2\in\gamma,
    \]
    to be an integration on $\gamma$.
    Then for each nonnegative integer $k$, there exists a constant $M_k$,
    depending only on $g \text{ and } \gamma$, such that 
    \[
        \left\lvert \dfrac{J_k(z_1,z_2)}{(z_1-z_2)^{k+1}}\right\rvert \leq M_k
    \]
    for all $z_1,z_2\in\gamma$ with $z_1\neq z_2$.
\end{lemma}

Next we show that if proper value is defined at $z_1=z_2$,
then $d_1(f|z_1,z_2)$, as a function of $z_1$, inherits the smoothness of $f'$,
which resolves our concern in the definition  of divided difference given in \eqref{eq:dvdef}
when $z_i=z_j$ for some $i\neq j$, assuming $f$ is smooth enough.
This property was mentioned in \cite{curtissjordan} where the proof was omitted.
For completeness, we give a rigorous proof below. 
\begin{lemma}
    \label{lm:lemma3}
    Let $f,\gamma,\phi$ be given as in Lemma \ref{lm:lemma1} with $n\geq 2$,
    and we further assume that $f^{(n-1)}\circ \phi$ is Lipschitz continuous.
    Then for any fixed $z_2=\phi(t_2)\in\gamma$,
    $|d_1^{n-1}(f|z_1,z_2)|$ is uniformly bounded in $z_1\in \gamma\backslash \{z_2\},$
    and for each integer $k$ with $0\leq k\leq n-2$, 
    we can assign a proper value at $t_1=t_2$ (i.e., $z_1=\phi(t_1)=z_2$) such that 
    $d_1^k(f|\phi(t_1),z_2)$ is absolutely continuous as a function of $t_1\in [0,1]$.
\end{lemma}
\begin{proof}
    For the following proof, we assume $k\in \{1,\dots,n-1\}$.

    Since $f^{(n-1)}\circ\phi$ is Lipschitz continuous,
    its derivative exists almost everywhere and is uniformly bounded.
    Lemma \ref{lm:lemma1} and Lemma \ref{lm:lemma2} then imply that 
    $|d_1^{k}(f|z_1,z_2)|$ is uniformly bounded in $z_1$ on $\gamma\backslash \{z_2\}.$
    In particular, $|d_1^{n-1}(f|z_1,z_2)|$ is uniformly bounded.

    With the uniform boundedness of $|d_1^{n-1}|$,
    we are able to define $d_1^{k-1}(f|z_1,z_2)$ at $z_1=z_2$
    such that $d_1^{k-1}(f|\phi(t_1),z_2)$ is absolutely continuous in $t_1\in [0,1]$.
    To do this, we first observe that 
    the total variation of $d_1^{k-1}(f|\phi(t_1),z_2)$ is uniformly bounded 
    on any subinterval of $[0,1]\backslash \{t_2\}$. 
    This implies that (cf.\cite[p.371, Ex.6]{titchmarsh1939book}),
    as $t_1$ approaches $t_2$ from either side,
    the limit of $d_1^{k-1}(f|\phi(t_1),z_2)$ exists,
    and assuming two limits coincide,
    if we define $d_1^{k-1}(f|\phi(t_2),z_2)$ to be equal to the limit,
    $d_1^{k-1}(f|\phi(t_1),z_2)$ is absolutely continuous as a function of $t_1$ on $[0,1]$.
    Thus if $t_2$ is an endpoint in $[0,1]$, 
    we can assign to $d_1^{k-1}(f|z_2,z_2)$ the unique limit as $t_1\to t_2$.
    If $t_2$ is an interior point in $[0,1]$,
    we shall prove that the two limits coincide as $t_1$ approaches $t_2$ from either side.
    In fact, from \eqref{eq:lemma1} and 
    the continuity of $f^{(k)}\circ\phi(t)$ in the assumption, 
    we deduce by using L'Hospital's rule that 
    \begin{equation}
    \label{eq:lopi}
    \begin{aligned}
        \lim_{t_1\to t_2}  d_1^{k-1}(f|\phi(t_1),\phi(t_2)) 
        &=
        \lim_{t_1\to t_2}
\dfrac{\int_{t_2}^{t_1} (\phi(t)-\phi(t_2))^{k-1} f^{(k)}\circ\phi(t)\phi'(t)dt}{(\phi(t_1)-\phi(t_2))^{k}} \\
        &= \lim_{t_1\to t_2} 
\dfrac{(\phi(t_1)-\phi(t_2))^{k-1} f^{(k)}\circ\phi(t_1)\phi'(t_1)}{k(\phi(t_1)-\phi(t_2))^{k-1}\phi'(t_1)} \\
    &= \dfrac{f^{(k)}(z_2)}{k}.
    \end{aligned}
    \end{equation}
    Hence by setting $d_1^{k-1}(f|z_2,z_2)$ to be equal to the limit above,
    we conclude that 
    $d_1^{k-1}(f|\phi(t_1),z_2)$ is absolutely continuous in $t_1\in [0,1]$. 
\end{proof}

The proof above implies that, 
with $d_1^{k-1}(f|z_2,z_2) (k=1,\dots,n-1)$ properly defined,
$d_1^{k-1}(f|\phi(t_1),z_2) $ will be absolutely continuous
in $t_1\in [0,1]$,
and $|d_1^{n-1}(f|z_1,z_2)|$ will be uniformly bounded in $z_1\in\gamma\{z_2\}$,
as long as the following two conditions are all satisfied:
\begin{enumerate}
    \item \label{cond1} 
    $f^{(n-1)}\circ\phi$ is absolutely continuous
    (consequently the representation formula 
    \eqref{eq:lemma1} in Lemma \ref{lm:lemma1} 
    holds for $d_1^{k-1}(f|z_1,z_2)\; (k=1,\dots,n)$);
    \item \label{cond2} 
    $|f^{(n)}\circ\phi(t)|$ is uniformly bounded in $[0,1]$ except on a set of measure zero 
    (hence Lemma \ref{lm:lemma2} can be applied to $f^{(n)}$).
\end{enumerate}

We next show that higher order divided differences can also be made continuous
by recursively verifying the two conditions above.

We set $g(z)=d_1(f|z,z_2).$
Lemma \ref{lm:lemma3} shows that $g^{(n-2)}\circ\phi(t)$ is absolutely continuous in $[0,1]$
and $|g^{(n-1)}(z)|$ is uniformly bounded on $\gamma\backslash \{z_2\}$.
Hence condition \ref{cond1} and condition \ref{cond2} are both satisfied,
and the absolute continuity of 
$d_1^{k-1}(g|\phi(t_1),z_3)$ ($k=1,\dots,n-2$)
in $t_1\in [0,1]$ follows,
with $d_1^{k-1}(g|z_3,z_3)$ properly defined .

Note that 
\[
\begin{aligned}
    d_2^{k-1}(f|z_1,z_2,z_3) &= 
    \frac{\partial^{k-1}}{\partial z_1^{k-1}} d_2(f|z_1,z_2,z_3) \\
    &= \frac{\partial^{k-1}}{\partial z_1^{k-1}} d_1(d_1(f|z,z_2)|z_1,z_3) 
    = d_1^{k-1}(g|z_1,z_3). 
\end{aligned}
\]
Therefore, we have established the absolute continuity of 
$d_2^{k-1}(f|z_1,z_2,z_3)$ ($k=1,\dots,n-2$),
as a function of $t_1=\phi^{-1}(z_1) \in [0,1]$,
as well as the uniform boundedness of $|d_2^{n-2}(f|z_1, z_2, z_3)|$.

Similarly,
we can then set $h(z) = d_2(f|z,z_2,z_3)$ and use Lemma \ref{lm:lemma3}
to deduce the absolute continuity of $d_3^{k-1}$ ($k=1,\dots,n-3$)
and uniform boundedness of $|d_3^{n-3}|$.

Therefore, by iteratively using Lemma \ref{lm:lemma3} 
to verify the two conditions mentioned above,
we arrive at the following theorem.
\begin{theorem}
    Let $f$ be a function defined on an admissible Jordan arc $\gamma$ with parametrization $\phi$
    such that $f^{(n-1)}$ ($n\geq 2$) exists everywhere on $\gamma$ 
    and $f^{(n-1)}\circ \phi (t)$ is Lipschitz continuous on $[0,1]$. 
Then for any integer $k$ with $1\leq k \leq n$,
$|d_k^{n-k}(f|\phi(t_1),
z_2,
\dots,
z_{k+1})|$ is uniformly bounded almost everywhere on $[0,1]$
as a function of $t_1$, and 
for $m=0,1\dots,n-k-1$ ,
$d_k^{m}(f|\phi(t_1),z_2,\dots,z_{k+1})$ 
is absolutely continuous in $t_1$ when proper value is defined at $z_1=\phi(t_1)=z_{k+1}$.
Moreover, the following equation holds as a generalization of Equation \eqref{eq:lemma1}.
\begin{equation}
\label{eq:relation}
    d_k^{m} = 
    \dfrac{\int_{z_{k+1}}^{z_1} (z-z_{k+1})^{m} d_{k-1}^{m+1}(f|z,z_2,\dots,z_k)dz}{(z_1-z_{k+1})^{m+1}},
    \quad  k=1,\dots,n, \;\; m=0,\dots,n-k,
\end{equation}
where $z_1\neq z_{k+1},\; d_n^0 = d_n, \text{ and } d_0^n = f^{(n)}$.
\end{theorem}

Equation \eqref{eq:relation}
will be the main tool that we use to derive the desired estimate 
in the next section.


\section{An upper bound for modulus of divided difference on a Jordan arc} 
\label{sec:bound}
In \cite{curtissjordan}, the boundedness of $|d_n|$ was obtained by 
a change of variable to convert the problem on a general Jordan curve to the problem 
on the unit circle.
However, 
this indirect approach makes 
the bound (which can be computed by following the proof in \cite{curtissjordan}) 
too pessimistic 
if the shape of the curve is not close to a circle.
In this section,
as opposed to \cite{curtissjordan}, 
we employ a direct approach to derive an explicit upper bound
that does not depend on the parametrization of $\gamma$
and that resembles the estimate in \eqref{eq:bound0}.

To start with,
we first compute upper bounds for $|d_1^k|\; (k=1,2,\dots,n-1)$.
\begin{lemma}
    \label{lm:d1k}
    Let $f$ be defined on an admissible Jordan arc $\gamma$
    such that $f^{(n+1)}$ exists and is continuous on $\gamma$.
    Then 
    \[
    |d_1^k(f|z_1,z_2)| \leq \dfrac{C_{\gamma,k}}{k+1},\quad \forall z_1\neq z_2\in\gamma,\;\;
        k = 1,2,\dots,n-1,
    \]
    where $C_{\gamma,k}$ is a nonnegative constant only depending on $f,\gamma,k.$
\end{lemma}
\begin{proof}
    The main idea of this proof is to use the representation in \eqref{eq:lemma1} 
    and then apply integration by parts.
    Indeed, suppose $z_1\neq z_2$, we deduce from \eqref{eq:lemma1} that,
    for $1\leq k \leq n-1$,
    \begin{equation}
    \label{eq:d1k}
    \begin{aligned}
    d_1^{k}(f|z_1,z_2) 
    &= \dfrac{\int_{z_2}^{z_1} (z-z_2)^{k} f^{(k+1)}(z)dz}{(z_1-z_2)^{k+1}} \\
    &= \dfrac{\frac{(z_1-z_2)^{k+1}}{k+1}f^{(k+1)}(z_1) - 
        \frac{1}{k+1}\int_{z_2}^{z_1} (z-z_2)^{k+1} f^{(k+2)}(z)dz}{(z_1-z_2)^{k+1}} \\
    &= \dfrac{f^{(k+1)}(z_1)}{k+1} -
     \dfrac{1}{k+1}\dfrac{\int_{z_2}^{z_1} (z-z_2)^{k+1} f^{(k+2)}(z)dz}{(z_1-z_2)^{k+1}}.
    \end{aligned}
    \end{equation}
    By using L'Hospital's Rule as in \eqref{eq:lopi},
    we find that 
    \[
    \begin{aligned}
        \lim_{z_1\to z_2} \dfrac{\int_{z_2}^{z_1} (z-z_2)^{k+1} f^{(k+2)}(z)dz}{(z_1-z_2)^{k+1}}
        &= \lim_{z_1\to z_2} \dfrac{(z_1-z_2)^{k+1} f^{(k+2)}(z_1)}{(k+1)(z_1-z_2)^k} \\
        &= \lim_{z_1\to z_2} (z_1-z_2)\frac{f^{(k+2)}(z_1)}{k+1}  = 0,
    \end{aligned}
    \]
    since $f^{(k+2)}$ is continuous on $\gamma$.
    Thus we have 
     \[
         M_{\gamma,k} := \sup_{z_1,z_2\in\gamma,z_1\neq z_2} 
    \left\lvert\dfrac{\int_{z_2}^{z_1} (z-z_2)^{k+1} f^{(k+2)}(z)dz}{(z_1-z_2)^{k+1}}\right\rvert < \infty.
     \]
    Hence it follows from \eqref{eq:d1k} that 
    \[
        |d_1^{k}(f|z_1,z_2)| \leq \dfrac{\sup_{z\in\gamma}|f^{(k+1)}(z)|+M_{\gamma,k}}{k+1}
        = \dfrac{C_{\gamma,k}}{k+1},
    \]
    where 
    \begin{equation}
    \label{eq:ck}
    \begin{aligned}
        C_{\gamma,k} &= \sup_{z\in\gamma}|f^{(k+1)}(z)|+M_{\gamma,k} \\
        &= \sup_{z\in\gamma}|f^{(k+1)}(z)| +
 \sup_{z_1,z_2\in\gamma,z_1\neq z_2} 
    \left\lvert\frac{\int_{z_2}^{z_1} (z-z_2)^{k+1} f^{(k+2)}(z)dz}{(z_1-z_2)^{k+1}}\right\rvert
    \end{aligned}
    \end{equation}
    only depends on $f,\gamma,k$, and $C_{\gamma,k}\to\sup_{z\in\gamma}|f^{(k+1)}(z)|$
    as $\text{diam}(\gamma)\to 0.$
\end{proof}

In order to estimate $|d_n|$ using \eqref{eq:relation},
we need the following elementary result.    
\begin{lemma}
    \label{lm:ineq}
    Let $\{I_{k,n}\}_{k,n=0}^\infty$ be a double sequence of nonnegative numbers satisfying 
    \[
    \begin{aligned}
    I_{1,n} &\leq \dfrac{C}{n+1}\\
    I_{k,n} &\leq \frac{1}{n+1}(I_{k-1,n+1} + L_k I_{k-1,n+2}),\quad k=2,\dots,\;\; n=0,1,\dots,
    \end{aligned}
    \]
    where for each $k$, $L_k$ is a nonnegative constant.
    Then 
    \begin{equation}
    \label{eq:ineq}
        I_{n,0} \leq \dfrac{C\prod_{k=2}^n (1+L_k)}{n!}, \quad \forall n \geq 2.
    \end{equation}
\end{lemma}
\begin{proof}
    We first define an upper bound $\hati_{k,n}$ of $I_{k,n}$ as follows.
    Let the double sequence $\{\hati_{k,n}\}_{k,n=0}^\infty$ be given by
    \[
    \begin{aligned}
    \hati_{1,n} &= \dfrac{C}{n+1}\\
    \hati_{k,n} &= \frac{1}{n+1}(\hati_{k-1,n+1} + L_k \hati_{k-1,n+2}),\quad k=2,\dots, n=0,1,\dots.
    \end{aligned}
    \]
    It is easy to verify by induction on $k$ that 
    $I_{k,n}\leq \hati_{k,n}$. 
    Thus it suffices to bound $\hati_{n,0}$.

    We observe that $\hati_{k,n+1}\leq \hati_{k,n}$.
    Indeed, for $k=1$, 
    $\hati_{1,n+1}=\frac{C}{n+2} \leq \frac{C}{n+1}= \hati_{1,n}$ by definition. 
    Assume the inequality holds with first index $k-1$.
    Then we see from the definition of $\hati_{k,n}$ and the hypothesis for $k-1$ that 
    \[
        \hati_{k,n+1} = \frac{1}{n+2}(\hati_{k-1,n+2} + L_k \hati_{k-1,n+3})
        \leq \frac{1}{n+1}(\hati_{k-1,n+1} + L_k \hati_{k-1,n+2}) = \hati_{k,n}.
    \]
    Hence the above induction implies that $\hati_{k,n+1}\leq \hati_{k,n}$ for all indices.
    
    Consequently, we have 
    \[
    \begin{aligned}
        \hati_{k,n} 
        &= \frac{1}{n+1}(\hati_{k-1,n+1} + L_k \hati_{k-1,n+2}) \\
        &\leq \frac{1}{n+1}(\hati_{k-1,n+1} + L_k \hati_{k-1,n+1}) 
        = \dfrac{1+L_k}{n+1} \hati_{k-1,n+1}.
    \end{aligned}
    \]
    Now we are able to estimate $\hati_{n,0}$ by induction below.
    \[
    \begin{aligned}
        \hati_{n,0} &\leq \dfrac{(1+L_n)}{1} \hati_{n-1,1} \\
        &\leq \dfrac{(1+L_n)(1+L_{n-1})}{2!} \hati_{n-2,2} \\
        &\leq \dots \leq \dfrac{\prod_{k=2}^n (1+L_k)}{(n-1)!} \hati_{1,n-1}
        = \dfrac{C\prod_{k=2}^n (1+L_k)}{n!}, \quad \forall n \geq 2.
    \end{aligned}
    \]
    The inequality \eqref{eq:ineq} then follows since $I_{n,0}\leq \hati_{n,0}.$ 
\end{proof}

We are now in a position to state the main result.
\begin{theorem}
\label{thm:theorem1}
    Let $f$ be a function defined on an admissible Jordan arc $\gamma$ 
    such that $f^{(n+1)}$ exists and is continuous on $\gamma$.
    Then 
    \begin{equation}
    \label{eq:upperbound}
        |d_n(f|z_1,\dots,z_{n+1})| \leq  
        \dfrac{C_\gamma\prod_{k=2}^n (1+L_k)}{n!} \leq
        \dfrac{C_\gamma \left( 1+\text{diam}(\gamma)  \right)^{n-1}}{n!}, \quad \forall n \geq 2,
    \end{equation}
    where $z_1,\dots,z_{n+1}\in\gamma$ are distinct, 
    $L_k = |z_1-z_{k+1}|$,
    $\text{diam}(\gamma) := \max_{u,v\in\gamma} |u-v|$ and 
    \[
        C_\gamma = \max_{1\leq k\leq n-1} C_{\gamma,k}
    \]
    with $C_{\gamma,k}$ defined in \eqref{eq:ck} independent of $z_1,\dots,z_{n+1}$ and the parametrization of $\gamma$.
    Furthermore,
    if $\gamma$ is an admissible Jordan curve 
    then \eqref{eq:upperbound} still holds.
\end{theorem}
\begin{proof}
    For a fixed set of points $z_1,\dots,z_{n+1}\in\gamma$,
    we define 
    \[
        I_{k,m} = |d_k^m(f|z_1,\dots,z_{n+1})|, \quad k=1,2,\dots,n,\; m=0,1,\dots,n-k.
    \]
    Let $C_\gamma = \max_{1\leq k\leq n-1} C_{\gamma,k}$ with $C_{\gamma,k}$ given in \eqref{eq:ck}.
    From Lemma \ref{lm:d1k}, we know that 
    \begin{equation}
    \label{eq:hypo1}
        I_{1,m} = |d_1^m| \leq \dfrac{C_{\gamma,m}}{m+1} \leq \dfrac{C_\gamma}{m+1}, \quad m=1,2,\dots,n-1.
    \end{equation}
    
    More generally, since $f^{(n+1)}$ is continuous on $\gamma$ (hence $f^{(n)}$ is Lipschitz continuous),
    the discussion in previous section shows that $d_k^m(f|\phi(t_1),z_2,\dots,z_k)$
    is absolutely continuous in $t_1\in [0,1]$ as long as $k+m\leq n,$
    from which integration by parts is justified.
    Therefore, based on \eqref{eq:relation},
    integration by parts as in \eqref{eq:d1k} yields that 
    \begin{equation}
    \label{eq:hypo2}
    \begin{aligned}
        I_{k,m} &= 
\left\lvert\dfrac{\int_{z_{k+1}}^{z_1} (z-z_{k+1})^{m} d_{k-1}^{m+1}(f|z,z_2,\dots,z_k)dz}{(z_1-z_{k+1})^{m+1}}\right\rvert \\
    &= \left\lvert\dfrac{d_{k-1}^{(m+1)}(f|z_1,\dots,z_k)}{m+1} -
     \dfrac{z_1-z_{k+1}}{m+1}
    \dfrac{\int_{z_{k+1}}^{z_1} (z-z_{k+1})^{m+1} d_{k-1}^{(m+2)}(f|z,z_2,\dots,z_k)dz}{(z_1-z_{k+1})^{m+2}} \right\rvert \\
    &= \left\lvert\dfrac{d_{k-1}^{(m+1)}(f|z_1,\dots,z_k)}{m+1} -
     \dfrac{z_1-z_{k+1}}{m+1} d_{k-1}^{m+2}(f|z_1,\dots,z_k) \right\rvert \\
    &\leq \dfrac{1}{m+1}(I_{k-1,m+1} + L_k I_{k-1,m+2}), \quad k=1,2,\dots,n,\; m = 0,1,\dots,n-k,
    \end{aligned}
    \end{equation}
    where $L_k := |z_1-z_{k+1}|$.
    It follows from \eqref{eq:hypo1}, \eqref{eq:hypo2} that 
    the assumptions in Lemma \ref{lm:ineq} are satisfied by $I_{k,m}$.
    With the help of \eqref{eq:ineq},
    we have
    \begin{equation}
        \label{eq:goodbound}
        \begin{aligned}
        |d_n(f|z_1,\dots,z_{n+1})| = |d_n^0|= I_{n,0} & \leq 
        \dfrac{C_{\gamma}\prod_{k=2}^n (1+L_k)}{n!} \\
    & \leq \dfrac{C_\gamma \left( 1+\text{diam}(\gamma)  \right)^{n-1}}{n!}, 
    \quad \forall n \geq 2,
        \end{aligned}
    \end{equation}
    which establishes \eqref{eq:upperbound}.

    Since
    both the integral given in \eqref{eq:inte} 
    and the derivative defined in \eqref{eq:deriv} are independent of 
    the parametrization of $\gamma$,
    we see that 
    $C_{\gamma,k}$ in \eqref{eq:ck} is independent of the parametrization 
    of $\gamma$.

    Suppose now $\gamma$ is a Jordan curve satisfying the hypothesis in 
    the claim
    and assume that we fix the orientation of the curve.
    Let $\gamma_0$ be an Jordan arc on $\gamma$ passing through 
    the nodes $z_1,\dots,z_{n+1}$.
    Then it is easily seen that the bound in \eqref{eq:goodbound} still holds 
    if we replace $\gamma$ in \eqref{eq:goodbound} by $\gamma_0$. 
    Furthermore, it follows immediately from \eqref{eq:ck} that
    $C_{\gamma_0} \leq C_{\gamma}$ if $\gamma_0\subset \gamma$,
    where $C_{\gamma_0}=\max_{1\leq k\leq n-1} C_{\gamma_0,k}$.
    Hence \eqref{eq:upperbound} follows, which completes the proof. 
\end{proof}

Note that divided difference is invariant under any permutation of nodes 
while the bound involving $L_k=|z_1-z_{k+1}|$ in \eqref{eq:upperbound} depends on 
the ordering of $z_1,\dots,z_{n+1}$.
We can then find a sharper bound by 
permuting the nodes to minimize the corresponding quantity in \eqref{eq:upperbound}
involving $L_k=|z_1-z_{k+1}|$.
\begin{corollary}
    Let $f,\gamma$ be given as in Theorem \ref{thm:theorem1}.
    For distinct nodes $z_1, \dots,$
    $ z_{n+1} \in \gamma$, we have 
    \begin{equation}
        \label{eq:sharper}
        |d_n(f|z_1,\dots,z_{n+1})| \leq  
        \dfrac{C_\gamma \min_{\sigma\in \mathcal{S}_{n+1}} \prod_{k=2}^n (1+|z_{\sigma(1)}-z_{\sigma(k+1)}|)}{n!}  ,\quad \forall n \geq 2,
    \end{equation}
    where $\mathcal{S}_{n+1}$ denotes the symmetric group of degree $n+1$ and
    $C_\gamma$ is given in Theorem \ref{thm:theorem1}. 
\end{corollary}

With an estimate of divided difference above,
an extension of \eqref{eq:error0} for polynomial interpolation error in complex plane 
readily follows.
\begin{theorem}
    \label{thm:theorem2}
    Let $f$ be a function defined on an admissible Jordan arc $\gamma$ 
    such that $f^{(n+1)}$ ($n\geq 2$) exists and is continuous on $\gamma$.
    If $p_n(z)$ interpolates $f$ at the $n+1$ nodes $z_1,\dots,z_{n+1}\in\gamma$,
    then we have the error estimate
    \begin{equation}
    \label{eq:error2}
        \left\lvert f(z)-p_n(z)\right\rvert \leq 
        \dfrac{C_\gamma \min_{\sigma\in \mathcal{S}_{n+2}} 
        \prod_{k=2}^{n+1} (1+|z_{\sigma(1)}-z_{\sigma(k+1)}|)}{(n+1)!}
        \prod_{k=1}^{n+1} \left\lvert{z-z_k}\right\rvert,
    \end{equation}
    where we have set $z_{n+2}=z$ ,  $C_\gamma := \max_{1\leq k\leq n-1} C_{\gamma,k}$
    with $C_{\gamma,k}$ defined in \eqref{eq:ck} independent of $z_1,\dots,z_{n+1}$,
    and $\mathcal{S}_{n+2}$ denotes the symmetric group of degree $n+2$.
    Furthermore,
    if $\gamma$ is an admissible Jordan curve   
    then \eqref{eq:error2} still holds. 
\end{theorem}
\begin{proof}
    This is an immediate result of Newton interpolation formula \eqref{eq:newton1} and 
    an estimate for 
    $|d_n(f|z_1,\dots,z_{n+1},z_{n+2})|$ using \eqref{eq:sharper}. 
\end{proof}


\bibliography{cdfeng}
\bibliographystyle{plain}
\end{document}